\documentclass[11pt,twoside]{article}

\usepackage{amsfonts}
\usepackage{amssymb}
\usepackage{amsmath}
\usepackage{graphicx}
\usepackage{hyperref}
\numberwithin{equation}{section}
\newtheorem{Prop}{\bf Proposition}[section]
\newtheorem{Cor}{\bf Corollary}[section]
\newtheorem{defn}{\bf Definition}[section]
\newtheorem{Rem}{\bf Remark}[section]
\newtheorem{Ex}{\bf Example}[section]

\begin{document}
\def \b{\Box}
\def \to{\mapsto}
\def \e{\varepsilon}
\def\i{\iota}

\begin{center}
{\Large {\bf  Almost groupoids and their substructures }}
\end{center}

\begin{center}
{\bf Mihai IVAN}
\end{center}

\setcounter{page}{1}
\pagestyle{myheadings}

{\small {\bf Abstract}. The aim of this paper is to present the main constructions of the substructures of an 
almost groupoid and to discuss their basic properties. The definitions and properties concerning these new algebraic constructions extend to almost groupoids, the corresponding well-known results for groupoids.
{}

\section{Introduction}
\indent\indent The concept of groupoid  from an algebraic point of view was first introduced  by H. Brandt in a 1926 paper \cite{brandt}. From this setting a groupoid (in the sense of Brandt)
can be thought of as a generalization of a group in which only certain multiplications are possible and it contains several  units elements. For basic results of Brandt groupoids, see (\cite{beglpt21, giv1b}).\\[-0.2cm]

The Brandt groupoids were generalized by C. Ehresmann in \cite{ehre50}. This new type of groupoid was called groupoid in the sense of Ehresmann. In addition,  C. Ehresmann added further structures (topological and differentiable as well as algebraic)
to Ehresmann  groupoids, thereby introducing them as a tool in topology and differential geometry. After the introduction of topological and differentiable  groupoids by Ehresmann  in the 1950's, they
have been studied by many mathematicians with different approaches  (\cite{giv1a, mack87, west71}).\\[-0.2cm]

The concepts of topological groupoid and Lie groupoid are defined using the definition of the groupoid in the sense of Ehresmann. For basic results  on topological groupoids, Lie groupoids, bundles of topological groupoids  and bundles of Lie groupoids, 
readers should consult (\cite{ivan24, miv2a,   miva05}).\\[-0.2cm]

The study of Brandt groupoids, Ehresmann groupoids, topological groupoids, Lie groupoids and Lie algebroids is motivated by their applications in many  branches of mathematics and engineering, namely in: algebra, analysis, geometry
 (\cite{avmp20, wein96};  category theory, topology, representation theory, algebraic topology (\cite{ehre50, igim11}); differential geometry, nonlinear dynamics, quantum  mechanics (\cite{goste06, gmio06, igmp11});
 study of some dynamical systems in Hamiltonian form and fractional dynamical systems using Lie groupoids and Lie algebroids (\cite{giom13, ivmi16, migi18, mlag, miva23, miva24, migo09}).\\[-0.2cm]

Another approach to the notion of a groupoid is that of a structured groupoid.  This  concept is obtained by adding another algebraic structure
such that the composition of the groupoid is compatible with the operation of the added algebraic structure. The most important types of structured group
 are  the following: group-groupoid  (\cite{brspen, ivan24})  and  vector space-groupoid (\cite{akiz18, miva13}).\\[-0.2cm]
  
An interesting  generalization of the group concept is that of almost groupoid.  Recently, the almost groupoid concept was introduced by Mihai Ivan in 2023 (\cite{miva23a}). It is very important
 to note that the class of almost groupoids is included in the class of groupoids. In other words, the category of almost groupoids is a subcategory of groupoids category.\\[-0.2cm]

The paper is organized as follows. Section 2 is devoted to giving basic definitions and some results related to Brandt groupoids. In Section 3 we give the definition of the almost groupoid from the algebraic point of view and
 their main properties are siscussed. In Section 4 we present some substructures in almost groupoids and prove several properties of them which generalize well-known results in group theory.

\section{ Some basic notions about Brandt groupoids}

In this section we reffer to  a new generalization of the notion of a  group, namely: the {\it almost groupoid} concept (\cite{miva23a}).  The definition of an almost groupoid from a purely algebraic point of view is 
 is given.
\begin{defn}  {\rm  (\cite{wein96})~A  (Brandt) \textit{groupoid $ {\cal G} $ over} $ G_{0}~$  is a pair $ ({\cal G}, {\cal G}_0) $ of nonempty sets such that $
{\cal G}_0\subseteq {\cal G}  $ endowed with two surjective maps $~\alpha,\beta : {\cal G} \rightarrow {\cal G}_{0},~$
a partially binary operation $~\mu :{\cal G}_{(2)}\rightarrow {\cal G}, (x,y)~\rightarrowtail~\mu(x,y):=x y,~$ where\\
$ {\cal G}_{(2)}=\{(x,y)\in {\cal G}\times {\cal G} | \beta(x)=\alpha(y)\}~$ and an injective map $~\iota: {\cal G} \rightarrow {\cal G}, x \rightarrowtail \iota(x):=x^{-1} $ satisfying the following properties:\\
$(1)~$({\it associativity})~$ (x y) z = x (y z),~$ in the sense that, if one side of the equation is defined so is the other one and then they are equal;\\
$(2)~$({\it identities}) $~(\alpha (x),x),~(x,\beta (x))\in {\cal G}_{(2)}~$ and $~ \alpha(x) x = x\beta(x) = x;$\\
$(3)~$({\it inverses}) $~(x^{-1},x),~(x,x^{-1})\in {\cal G}_{(2)}~$ and $~x^{-1} x = \beta(x),~ x x^{-1} = \alpha(x).~$}
\end{defn}
The elements of $~{\cal G}_{(2)}~$ are called {\it composable pairs} of $~{\cal G}.~$ The element $~\alpha(x)~$ (resp. $~\beta(x)~$) is the {\it left unit}  (resp. {\it right unit})
of $~x\in {\cal G}.~$ The subset $~{\cal G}_{0} = \alpha({\cal G}) = \beta({\cal G})~$ of $~{\cal G}~$  is called the {\it unit set} of $~{\cal G}.~$\\[-0.2cm]

 A groupoid $ {\cal G} $ over $~{\cal G}_{0} $ with the \textit{structure functions} $ \alpha $ ({\it source}), $ \beta $ ({\it target}), $ \mu $ ({\it
partial multiplication}) and $\iota$ ({\it inversion}), is denoted by $ ({\cal G}, \alpha, \beta, \mu, \iota, {\cal G} _0) $ or $ ~({\cal G}, {\cal G}_{0}).$\\[-0.2cm]

 Whenever we write a product in a given a groupoid $~({\cal G}, {\cal G}_{0}),~$  we are assuming that it is defined.\\[-0.2cm]

A group $~{\cal G}~$ having $~e~$ as unit element is just a groupoid over $~\{ e\},~$ and conversely, every groupoid with one unit element is a group.\\[-0.2cm]

\markboth{Mihai IVAN}{Almost groupoids and their substructures}

If $ ({\cal G}, {\cal G}_{0}) $ is a groupoid and $ u\in {\cal G}_{0}$, then $~{\cal G}(u):=\alpha^{-1}(u)\cap\beta^{-1}(u)$ is a group under the
restriction of  $\mu $ to $ {\cal G}(u),$ called the {\it isotropy group at} $u$ of ${\cal G}$.\\[-0.2cm]

In a groupoid $ ({\cal G}, {\cal G}_{0}), $ the map $~(\alpha, \beta): {\cal G} \to {\cal G}_{0}\times G_{0}~$ defined by\\
$(\alpha, \beta)(x):=(\alpha(x), \beta(x)), (\forall) x\in{\cal G}  $ is called the \textit{anchor map} of $ {\cal G}.$ A  groupoid is said to be {\it
transitive}, if its anchor map is surjective.\\[-0.2cm]

In the following proposition we summarize some properties of the structure functions of a groupoid obtained directly from definitions.
\begin{Prop} {\rm \cite{giv1b})}
~For any groupoid $~({\cal G}, \alpha, \beta, \mu, \iota ,{\cal G}_{0})~$ the following assertions hold:\\
$(i)~~~\alpha(u) = \beta(u) =\iota(u)= u ~$ and $~u\cdot u = u~$ for all $~u\in {\cal G}_{0};~$\\
$(ii)~~~\alpha(x y) = \alpha(x)~$ and $~ \beta(x y ) = \beta(y), ~(\forall)~ (x,y)\in {\cal G}_{(2)};~$\\
$(iii)~~~ \alpha \circ \iota = \beta, ~~ \beta\circ \iota = \alpha ~$ and $~ \iota \circ \iota = Id_{{\cal G}};~$\\
$(iv)~~~(x,y)\in {\cal G}_{(2)}~~\Longrightarrow~~(y^{-1},x^{-1})\in {\cal G}_{(2)}~~$ and $~~(x y )^{-1} = y^{-1} x^{-1};~$\\
$(vi)~~~$ For each $~u\in {\cal G}_{0},~$ the set  ${\cal G}(u)=\alpha^{-1}(u)\cap\beta^{-1}(u)=\{ x\in {\cal G}~|~\alpha(x)=\beta(x)=u~\}$ is a group under the restriction
of $\mu~$ to $~{\cal G}(u),$ called the {\bf isotropy group  of $~{\cal G} $ at} $u;$\\
$(vii)~~~$  For each $~x\in {\cal G}, $ the isotropy groups $~{\cal G}(\alpha(x))~$ and $~{\cal G}(\beta(x))~$ are isomorphic;\\
$(viii)~~$ If $~({\cal G}, {\cal G}_{0})~$  is transitive, then the isotropy groups $~{\cal G}(u),~u\in {\cal G}_{0}$ are isomorphes.
\end{Prop}
\begin{Ex}~{\rm (\cite{ivan25a}) Let $~{\bf R}^{\ast}~$ be the set of nonzero real numbers and $~a, b\in {\bf R}^{\ast}~$ such that $~a b=1. $ Consider the sets  $ {\cal G}= {\bf R}^{\ast}\times {\bf R}^{\ast}: = {\bf R}^{\ast 2}~ $ and 
$ {\cal G}_{0} =\{ (a, a x)\in {\cal G}~|~(\forall) x\in  {\bf R}^{\ast}\}. $  It is easy to see that $~ y=ax ~$ if and only if $ x =b y $ for all $~(x, y)\in {\cal G}. $ Then
$ ({\cal G}, \alpha, \beta,  \odot, \iota, {\cal G}_{0})~$ is a groupoid, where the set $~{\cal G}_{(2)}~$ and  its structure functions are given by:\\
$\alpha : {\cal G}~\rightarrow~ {\cal G}_{0}, (x,y)~\rightarrowtail \alpha (x,y):=(x,ax),~\beta : {\cal G}~\rightarrow~ {\cal G}_{0}, (x,y)~\rightarrowtail \beta (x,y):=(by,y),~$;\\
$ {\cal G}_{(2)} = \{ ((x,y), (z, u))\in {\cal G} \times {\cal G}~|~ \beta (x,y) =\alpha (z, u)\} = \{ ((x,y), (z, u)) \in {\cal G} \times {\cal G} ~|~ z=by \};$\\
$(x,y)~\odot~(by,u):=(x, u),~ (\forall)~x, y, u \in {\bf R}^{\ast}~~$ and $~~\iota(x,y):=(by, ax), (\forall)~x, y\in {\bf R}^{\ast}. $
It is easy to verify that the conditions $ (1) - (3) $ of Definition 2.1 hold. First, we consider $~ (x, y), (z, u), (v, w) \in {\cal G}.~$   The product
$ (x, y)\cdot (z, u) \cdot (v, w) $ is defined if and only if $~z=by~$ and $~v=bu.~$ We have:\\
$ (1)~~~((x, y)\odot (by, u))\oplus (bu, w)=((x, u)\odot (bu, w) = (x,w)= (x, y)\odot ((by, u)\oplus (bu, w)); $\\
$ (2)~~~\alpha (x, y)\odot (x, y))= (x, ax)\odot (x, y)= (x, t)\odot (bt, y)=(x, y) ~$ and $~(x, y)\odot \beta(x,y)=(x,y)\odot (by, y)= (x,y);$\\
$ (3)~~~(x, y)\odot \iota(x, y)= (x, y)\odot (by, ax)= (x, ax)=\alpha (x,y)~$ and $~\iota(x,y)\odot (x,y)= (by, ax)\odot (x,y)=(by, t)\odot (bt,y)=(by,y)=\beta(x, y).$
This groupoid is denoted with $~{\bf R}^{\ast 2}(a,b).~$ Therefore, for each pair $~(a,b)\in {\bf R}^{\ast 2} ~$   satisfying relation $~ ab=1,~$ an example of groupoid is obtained in this way. }\\[-0.4cm]
\end{Ex}
\begin{defn}~ {\rm  A non-empty subset $~H~$ of a $~{\cal G}_{0}~$-groupoid $~{\cal G}~$ is called {\it subgroupoid} of $~{\cal G}~$ if it is closed under multiplication (when it is defined)  and inversion, i.e.
the following conditions hold:\\
$(i)~~~(\forall) ~x,y\in H~$ such that $~x y~$ is defined, we have $~x y\in H;~$\\[-0.4cm]
$(ii)~~(\forall) ~x\in H ~~\Rightarrow ~~x^{-1}\in H.$}
\end{defn}
Note that from the condition (ii) of Definition 2.2, we deduce that $~\alpha(h) \in H~$ and $~\beta(h)\in H,~(\forall) ~h\in H.~$ If $~\alpha(H) = \beta(H) = {\cal G}_{0},~$
then $~H~$ is called a {\it wide subgroupoid}.
\begin{defn}~ {\rm $~(i)~$ Let $~( {\cal G}, \alpha, \beta, \mu, \iota , {\cal G}_{0})~$ and $~( {\cal G}^{\prime}, \alpha^{\prime}, \beta^{\prime}, \mu^{\prime}, \iota^{\prime}, {\cal G}_{0}^{\prime})~$  be two groupoids.
A  {\it morphism between these groupoids} is a pair $~(f, f_{0})~$ of maps, where $~f: {\cal G} \longrightarrow {\cal G}^{\prime}~$ and $~f_{0} : {\cal G}_{0}~\to~{\cal G}_{0}^{\prime},~$ such that the
following two conditions are satisfied:\\
$(i)~~~ f(\mu(x,y)) = \mu^{\prime}(f(x), f(y)),~~~(\forall) ~(x,y)\in {\cal G}_{(2)};~$\\
$(ii)~~~ \alpha^{\prime}\circ f = f_{0}\circ \alpha ~$ and $~ \beta^{\prime}\circ f = f_{0}\circ \beta.~$\\
$(ii)~$ A morphism of groupoids $~(f, f_{0})~$ is said to be {\it isomorphism of groupoids}, if $~f~$ and $~f_{0}~$ are bijective maps.}\\[-0.7cm]
\end{defn}
\begin{Rem}
{\rm  The fundamental  concepts related to Brandt groupoids and Ehresmann groupoids  has been defined and studied in a series of papers, for instance:
substructures, normal subgroupoids, quotient groupoid, morphisms  and strong morphisms groupoids, construction of some ways of building up new groupoids from old ones  (\cite{avmp20, beglpt21, ivan24, miva13}
 A special field in groupoid theory is dedicated to the study of finite groupoids.  A remarkable example of finite groupoid  is the symmetric groupoid    $~{\cal S}_{n}~$  (\cite{miva02, miv3d}).
This groupoid   play an important role in the study of finite groupoids, since by Cayley's theorem every finite groupoid of degree $~n~$ is isomorphic to some subgroupoid of $~{\cal S}_{n}.~$ (\cite {avmp20, giv1b})};\\[-0.7cm]
\end{Rem}

\section{Amost groupoids: definition and basic properties}

In this section we refer to  new generalization of the notion of group, namely: the {\it almost groupoid}  (\cite{miva23a}). The definition of an almost groupoid from a purely algebraic 
point of view is given.

\begin{defn}(\cite{miva23a})
{\rm   An {\it almost groupoid $ G$ over} $ G_{0}$   ({\bf in the sense of Brandt)} is a pair
$( G, G_{0} )$ of nonempty sets such that $~G _0\subseteq G,~ $ endowed with a surjective map
$\theta: G \rightarrow G_{0}$, a partially binary operation $~ m : G_{(2)}\rightarrow G ,~(x,y)\longmapsto  m ( x,y):=x\cdot y,~$ where\\
$~~~~~~~~~~~~~~~~~~~~~~~~~~~~ G_{(2)}:= \{ (x,y)\in G\times G | \theta(x) = \theta(y) \}$ \\
and a map $\iota:G \rightarrow G ,~~x\longmapsto \iota(x):=x^{-1},$ satisfying the following properties:\\
$({\bf AG1})~$ ({\it associativity})$~~(x\cdot y)\cdot z=x\cdot (y\cdot z)~$ in the sense that if
one of two products $ (x\cdot y)\cdot z $ and $ x\cdot (y\cdot z) $ is defined,
then the other product is also defined and they are equals;\\
$({\bf AG2})~$ ({\it units}): for each
$ x\in G~\Rightarrow~(\theta(x),x),~(x,\theta( x))\in G _{(2)}$
and  we have\\ $ \theta(x)\cdot x= x \cdot \theta(x)= x;$\\
$({\bf AG3})~$ ({\it inverses}): for each $ x\in G~\Rightarrow~(x, x^{-1}),~( x^{-1}, x) \in G_{(2)} $
 and we have\\ $ x\cdot x^{-1}= x^{-1}\cdot x= \theta(x).$}
\end{defn}

If $ G $ is an almost groupoid over $ G_{0}, $  we will sometimes write $x\cdot y $ or $ x y$ for $ m(x,y). $ Also, the set $ G_{(2)} $ is called  the {\it set of composable pairs} of $ G. $

 An almost groupoid $ G $ over $ G_{0} $  with the \textit{structure functions} $ \theta
$ ({\it units map}),$ m $ ({\it multiplication}) and $ \iota $ ({\it inversion}), is denoted by $ (G, \theta,  m, \iota, G_{0})~$  or $~(G, G_{0}).~ G_{0} $ is called the {\it units set} of $ G. $.  Whenever we write a product in a given almost groupoid, we are
assuming that it is defined.

In view of Definition 3.1, if $ x,y, z\in G, $ then:

 $(i)~~$ the products  $~~x\cdot y\cdot z~$ and $~ (x\cdot y)\cdot z ~~$ {\it  are  defined} $~~~ \Leftrightarrow ~~~ \theta(x)=\theta (y) =\theta (z). $

$(ii)~~ \theta(x)\in G_{0} $ is {\it the unit} of $ x $ and  $ x^{-1}\in G $ is the {\it inverse} of $ x.$\\

An almost groupoid $ (G, \theta, m, \iota, G_{0}) $ can be represented in the form of the following diagram:

$$~~~~~~~~~~~~~~~~~~~\begin{array}{lllll}
G\times G \supset G_{(2)} & \stackrel{m}{\longrightarrow } & G & \stackrel{\iota}{\longrightarrow }~~ G \\
 &  & \downarrow \theta  &  &   \\
 &  &  G_{0} & &
\end{array}~~~~~~~~~~~~~~~~~~~~~~~~~~~~~~~~~~~~~~~~~~~ (1)$$

The basic properties of the almost groupoids are given in the following two propositions.
\begin{Prop}   (\cite{miva23a}) If $ (G,\theta,  m, \iota, G_{0}) $ is an almost groupoid, then:

$(i)~~~\theta(u)=u,~(\forall)u\in G_{0}.$

$(ii)~~u\cdot u=u ~$ and $~\iota (u)=u,~(\forall)u\in G_{0}.$

$(iii)~\theta(x\cdot y) =\theta (x), ~(\forall ) ( x,y ) \in G _{(2)}.$

$(iv)~\theta ( x^{-1}) =\theta ( x)~$ and $~\theta(\theta(x))= \theta(x), ~~  (\forall) x\in G.$

$(v) $  If $ (x,y), (y,z)\in G_{(2)}~~\Rightarrow~~(x\cdot y, z), (x, y\cdot z) \in G_{(2)}~ $  and
$~(x \cdot y)\cdot z = x \cdot (y\cdot z).$
\end{Prop}

\begin {Prop}~{\it If $ (G,\theta,  m, \iota, G_{0}) $ is an almost groupoid, then:

$(i)~~$	{\bf (uniqueness of  the units)}:\\
(1) $~(x,y)\in G_{(2)} $ and $~x\cdot y = y~~\Rightarrow~~ x = \theta (y); $\\
(2) $~(x,y)\in G_{(2)} $ and $~x\cdot y = x~~\Rightarrow~~ y = \theta (x). $

$ (ii)~ $	{(\bf uniqueness of the inverse)}:\\
 if $ (x,y)\in G_{(2)},~~ x\cdot y =\theta (x)~$ and $~y\cdot x = \theta(x)~~\Rightarrow ~~y= x^{-1}.$

$(iii) $ {\bf(cancellation laws)} If  $ (x, y_{1}), (x, y_{2}), (y_{1}, z ), (y_{2}, z )\in G_{(2)}, $ then:\\
(1) $~~x\cdot y_{1}= x\cdot y_{2}~~\Rightarrow~~ y_{1}=y_{2}; $\\
(2) $~~ y_{1}\cdot z=y_{2}\cdot z~~\Rightarrow y_{1}=y_{2}.$

$(iv)~~ $  if $~(x,y)\in G_{(2)},$ then $~(y^{-1},x^{-1})\in
G_{(2)}~ $ and $~(x\cdot y)^{-1}=y^{-1}\cdot x^{-1}.$

{\it(v)}$~~ (x^{-1})^{-1}=x, ~~(\forall) x\in G.$

{\it(vi)} $~~ x^{-1}\cdot(x\cdot y)=y~~$ and $~~(x\cdot y)\cdot y^{-1}=x,$ for all $~(x,y)\in G_{(2)}.$}
\end{Prop}

{\bf Proof.} {\it(i)}(1). From $ (x, y)\in G_{(2)} $ it follows $ \theta (x)=\theta(y). $  We have:\\
$  x\cdot y = y ~|\cdot y^{-1}~ \Rightarrow~ (x\cdot y)\cdot y^{-1} = y\cdot y^{-1}~ \stackrel{(AG1)}\Rightarrow~ x\cdot (y\cdot y^{-1})= y\cdot y^{-1}~\stackrel{(AG3)}\Rightarrow~\\
 x\cdot \theta(y)=\theta(y)~\Rightarrow~  x\cdot \theta(x)=\theta(y)~\stackrel{(AG2)}\Rightarrow~ x= \theta(y).$\\
 Hence, the assertion (1) holds. Similarly, we prove that the assertion (2) holds.

 {\it(ii)} From $(x,y)\in G_{(2)} $ it follows $ \theta(x)=\theta(y). $  We have successively:\\
 $  y = \theta(y) \cdot y = \theta(x)\cdot y = (x^{-1}\cdot x)\cdot y  \stackrel{(AG1)}{=}~ x^{-1}\cdot (x\cdot y)= x^{-1}\cdot \theta (x)= x^{-1}\cdot \theta (x^{-1})= x^{-1}. $

 {\it(iii)} For example, we prove the cancellation law at left. From $(x,y_{1}), (x, y_{2}\in G_{(2)} $ it follows $ \theta(x) = \theta(y_{1})= \theta(y_{2}). $ We have successively:\\
 $ x\cdot  y_{1} = x\cdot  y_{2} ~|\cdot x^{-1}~\Rightarrow~ x^{-1}\cdot (x\cdot  y_{1}) = x^{-1}\cdot (x\cdot  y_{2})~\stackrel{(AG1)}\Rightarrow~
 (x^{-1}\cdot x) \cdot  y_{1} = (x^{-1}\cdot x)\cdot  y_{2})\\~\Rightarrow~ \theta(x)\cdot y_{1} = \theta(x)\cdot y_{2}~\Rightarrow~ \theta(y_{1})\cdot y_{1} = \theta(y_{2})\cdot y_{2}~ \stackrel{(AG2)}\Rightarrow~ y_{1} = y_{2}. $

{\it(iv)} Let $(x,y)\in G_{(2)}.$ Then $ \theta (x)=\theta(y). $ By the
assertions $ (iii) $ and $ (iv), $ we have $ \theta (y^{-1})=\theta(y)=\theta(x) = \theta (x^{-1}). $
 Since $ \theta (y^{-1})=\theta (x^{-1}), $ it follows that $~(y^{-1},x^{-1})\in
G_{(2)}. $ Also, $ \theta (y^{-1}\cdot x^{-1})=\theta(y^{-1})=\theta(y) = \theta(x). $ From
$~\theta (y^{-1}\cdot x^{-1})=\theta(x) = \theta (x\cdot y), $ it follows  $~(y^{-1}\cdot x^{-1}, x\cdot y)\in
G_{(2)}. $  Similarly, we get that $~(x\cdot y, y^{-1}\cdot x^{-1})\in G_{(2)}.$

Using $ {\bf (AG1)}, {\bf (AG2)} $ and $ {\bf (AG3)},$  we have succesively:\\
$(x\cdot y)\cdot (y^{-1}\cdot x^{-1}) = x\cdot (y\cdot y^{-1})\cdot x^{-1} = x\cdot \theta (y) \cdot x^{-1} =
(x\cdot \theta (x))\cdot x^{-1} = x\cdot x^{-1} =\theta (x). $

Hence, we have:\\
$(a)~~(x\cdot y)\cdot (y^{-1}\cdot x^{-1}) = \theta (x). $

Similarly, we prove:\\
$(b)~~(y^{-1}\cdot x^{-1})\cdot (x\cdot y) = \theta (x). $

From $(a)$ and $(b),~ y^{-1}\cdot x^{-1} $ is the inverse of $ x\cdot y, $  that is
$~(x\cdot y)^{-1}=y^{-1}\cdot x^{-1}. $

{\it(v)} From  $~x^{-1}\cdot x= \theta(x) =\theta(x^{-1})~$ and $~ x \cdot x^{-1} = \theta(x^{-1}) $ and the uniqueness of the inverse it follows that
$ x $ is the inverse of $ x^{-1}, $ i.e.  $~~ (x^{-1})^{-1}=x.$

{\it(vi)} The two equalities are verified by direct calculation. For example, for $ (x,y)\in G_{(2)} $ we have $~\theta(x)=\theta(y). $  The product $ x^{-1} \cdot (x\cdot y) $ is defined, since
$ \theta(x^{-1}) =\theta(x)=\theta(x\cdot y). $  Then $~ x^{-1}\cdot(x\cdot y)= (x^{-1}\cdot x)\cdot y=\theta(x)\cdot y = \theta(y)\cdot y =y.\hfill\Box$

\begin{Cor}~ {\it If $ (G,\theta,  m, \iota, G_{0}) $ is an almost groupoid, then:

$(i)~~~~~~ \theta(\theta(x))= \theta(x), ~~(\forall) x\in G. $

$(ii)~~~~$  If $~ x,y\in G,~$  then:  $~~ x\cdot y ~~$ is defined $~~~\Leftrightarrow~~~ y\cdot x~~$  is defined.

$(iii)~~~ (\forall)~ a\in G~~\Rightarrow ~~$ the products $~ a^{2}:= a\cdot a, ~ a^{3}:= a^{2}\cdot a~ $ are defined.}

{\bf Proof.} $(i)~$ For each $ x\in G, $ we have $~ \theta(\theta(x)) \stackrel{(AG3)}{=}~\theta (x\cdot x^{-1}) \stackrel{(3.1(iii))}{=}~  \theta(x).$

$(ii)- (iii)~$ These staments follow from definition and Proposition 3.1.\hfill$\Box$
\end{Cor}

\begin{Prop}~ {\it  Let $ (G,\theta,  m, \iota, G_{0}) $  be an almost groupoid. The structure functions of $ G $ have the following properties:

$~~\theta \circ \iota = \theta~~
\hbox{and}~~ \iota \circ \iota = Id_{G}.$}
\end{Prop}

{\bf Proof.}  The first relation is a consequence of equality (iv) in Proposition 3.1. Indeed,  $~\theta(x^{-1}) = \theta(\iota(x))= (\theta\circ\iota)(x). $ But,
$~\theta(x^{-1}) = \theta(x).$ Hence, $~(\theta\circ\iota)(x) = \theta(x),  (\forall) x\in G.~$ Also, we have $ (\iota\circ\iota)(x)=\iota(\iota(x))=\iota(x^{-1}) = (x^{-1})^{-1} =x. $
Then $~\iota \circ \iota = Id_{G},$ since $~ (\iota\circ\iota)(x) = Id_{G}(x), (\forall) x\in G.\hfill\Box$

\begin{Prop}~ {\it Let $ (G,\theta,  m, \iota, G_{0}) $  be an almost groupoid. For each $ u \in G_{0}, $ the set $ G(u):=\theta^{-1}(u)= \{ x\in G | \theta(x)= u \} $ is
a group.}
\end{Prop}

{\bf Proof.} By Proposition 3.1(i) we have that $\theta(u)=u. $ then $ u\in G(u). $ If $ x,y\in  G(u) $ then $ \theta(x)=\theta (y) = u $ and so  the product $ x\cdot y $ is defined.
According to Proposition 3.1(iii) implies that $ \theta (x\cdot y)=\theta (x) = u, $ then  $ x\cdot y\in G(u). $ Also, if $ x\in G(u) $ then the products $ x\cdot u $ and $ u\cdot x $ are defined
and $ x\cdot u = x\cdot \theta(x) x = x ~$ and $ u\cdot x =\theta (x)\cdot x = x ,~$ that is, $ u $ is the unit element of  $ G(u). $  Finally, let $x\in G(u). $ Then by Proposition 3.1(iv),
$ \theta(x^{-1})=\theta(x)=u. $ Hence $ x^{-1} \in G(u)\subset G ~$ and $ x\cdot x^{-1} = x^{-1}\cdot x =u. $ Therefore, we conclude that $ G(u) $  endowed with the restriction of
partial multiplication $ m $ at $ (G(u)\times G(u)\cap G_{(2)}~$ is a group having $ u $ as unity.\hfill$\Box$

The group $ G(u) $ for $ u\in G_{0} $ is called the {\it isotropy group} of $ G $ at $ u.$

\begin{Rem} {\rm
Each  almost groupoid  $~(G, \theta, m, \iota, G_{0}) $  is a  groupoid for which the structure functions  $ \alpha $ (source) and $\beta $ (target) are equal to $~\theta. $ Clearly,
it is not true in general that every groupoid is  an almost groupoid.}
\end{Rem}

{\begin{defn}~{\rm   An {\it almost groupoid $ G$ over} $ G_{0}$ is called {\it abelian} or {\it commutative} if the isotropy group $ G(u) $ is abelian for all $ u\in G_{0}. $}
\end{defn}
\begin{Ex}~{\rm $(i)~$ A group $ G $ having $ e $ as unity, is an almost groupoid over   $ \{e
\}$ with respect to structure functions: $~\theta (x): = e, ,(\forall)x\in G; ~G_{(2)}= G\times G,~ m (x,y):= xy,~ (\forall) x,y\in G
  $ and $ \iota:G \to G,~\iota(x):= x^{-1}, (\forall)x\in G.$

$(ii)~$ A nonempty set $ G_{0} $ may be regarded to be an
almost groupoid over $ G_{0}, $ called the {\it null almost groupoid} associated
to $ G_{0}. $ For this, we take $~ \theta=\iota = Id_{G_{0}} $ and
 $ u\cdot u = u, ~(\forall)  u\in G_{0}.$}
\end{Ex}
\begin{Ex}~{\rm Let $ G = {\bf R}\times {\bf R} = {\bf R}^{2}~ $ and $ G_{0} = {\bf R}\times \{0\}. $  Then
$ (G, \theta, \oplus, \iota, G_{0})~$ is an almost groupoid, where the set $~G_{(2)}~$ and  its structure functions are given by:\\
$\theta : G={\bf R}^{2}~ \rightarrow~ G_{0}, (a,b)~\rightarrowtail \theta (a,b):=(a,0);$\\
$ G_{(2)} = \{ ((a,b), (c, d))\in G \times G ~|~ \theta (a,b) =\theta (c, d)\} = \{ ((a,b), (c, d))\in G \times G ~|~ a=c \};$\\
$(a,b)~\oplus~(a,d):=(a, b+d),~ (\forall)~a,b,d \in {\bf R}~~$ and $~~\iota(a,b):=(a,-b), (\forall)~a,b\in {\bf R}. $

It is easy to verify that the conditions $ ({\bf AG1}) - {\bf AG3}) $ of Definition 3.1 hold. First, we consider $~ x, y, z\in G ,~$ where $~x=(a_{1}, b_{1}), y=(a_{2}, b_{2}), z=(a_{3}, b_{3}). $  The product $ x\cdot y \cdot z $ is defined if and only if $~a_{1} = a_{2} = a_{3}.$ We have:\\
$ (1)~~~(x\oplus y)\oplus z= ((a_{1}, b_{1})\oplus (a_{1}, b_{2}))\oplus (a_{1}, b_{3})= (a_{1}, b_{1}+b_{2})\oplus (a_{1}, b_{3})= (a_{1}, b_{1}+ b_{2}+ b_{3}) =x\oplus (y\oplus z); $\\
$ (2)~~~x\oplus \theta(x)= (a_{1}, b_{1})\oplus \theta(a_{1}, b_{1}))= (a_{1}, b_{1})\oplus(a_{1}, 0)=(a_{1}, b_{1}+0)= x =\theta(x)\oplus x;$\\
$ (3)~~~x\oplus \iota(x)= (a_{1}, b_{1})\oplus \iota(a_{1}, b_{1}))= (a_{1}, b_{1})\oplus(a_{1}, -b_{1})=(a_{1}, 0)= \theta(x)= \iota(x)\oplus x.$}\\[-0.2cm]
 \end{Ex}  
\begin{Ex} {\rm (\cite{miva23a})~ For $~a, k \in {\bf R}$ consider the matrix $~ A(a, k) =\left (\begin{array}{cc}
a & k  a\\
0 & 1\\
\end{array}\right). $ Let $~G =\displaystyle{ \{}~ A(a, k) ~|~ a, k\in {\mathbb{\ R}}, a\neq 0 \displaystyle{\}}~$ and 
$~G_{0} = \{~ A(1, k) ~|~ k \in {\mathbb { R}}\}.~$
  Then $ (G, \theta, \odot, \iota, G_{0})~$ is an almost groupoid.

\medskip
  The structure functions $ \theta $ and $ \iota $ are defined by:\\

  $\theta : G~\rightarrow~ G_{0}, A(a,k)~\rightarrowtail \theta (A(a,k)):= A(1,k)=\left (\begin{array}{cc}
1 & k \\
0 & 1\\
\end{array}\right);$\\

$\iota : G \rightarrow~ G, A(a,k)~\rightarrowtail \iota(A(a,k)):= A(a^{-1},k)=\left (\begin{array}{cc}
a^{-1} & k \\
0 & 1\\
\end{array}\right).$

\smallskip\noindent
The set of composable pairs $~ G_{(2)}~$ is given by:

\medskip\noindent
$~~~~~ G_{(2)} = \{ ( A(a_{1},k_{1}), A(a_{2},k_{2}))\in G \times G ~|~ \theta (A(a_{1},k_{1})) =\theta (A(a_{2},k_{2}))\} $

$~~~~~~ =\{ ( A(a_{1},k_{1}), A(a_{2},k_{2}))\in G \times G ~|~ k_{1} = k_{2}\}.$

\medskip\noindent
 The product $ \odot $  is defined by:  $ A(a_{1},k_{1})\odot  A(a_{2},k_{1}):= A(a_{1} a_{2},k_{1}),~$ that is:\\
 [2mm] $ \left (\begin{array}{cc}
a_{1} & k_{1}  a_{1}\\
0 & 1\\
\end{array}\right) \odot \left (\begin{array}{cc}
a_{2} & k_{1}  a_{2}\\
0 & 1\\
\end{array}\right) := \left (\begin{array}{cc}
a_{1} a_{2} & k_{1} a_{1}  a_{2}\\
0 & 1\\
\end{array}\right), (\forall)~ k_{1}\in {\bf R}, a_{1} a_{2}\neq 0. $\\

It is easy to verify that the conditions $ ({ AG1}) - ( AG3) $ of Definition 3.1 hold. For this, we consider $~ A, B, C\in G ,~$ where $~A= A(a,k), B=A(a_{1},k_{1}),C=A(a_{2},k_{2}). $  The product $ A\odot B \odot C $ is defined if and only if $~k = k_{1} = k_{2}.$ We have:

\medskip\noindent
$ ({AG1})\, (associativity)$ $A\odot B)\odot C= (A(a,k)\odot A(a_{1},k))\odot A(a_{2}, k)=$ 

\hspace*{6mm}$A(aa_{1}, k)\odot A(a_{2}, k)= A(aa_{1} a_{2}, k) =A\odot (B\odot C); $

\smallskip\noindent
$ ({ AG2})\,(units)$ $A\odot \theta(A)= A(a, k)\odot \theta (A(a,k))= A(a, k)\odot A(1,k)=$

\hspace*{6mm}$A(a,k)= A=\theta(A)\odot A ;$

\smallskip\noindent
$ ({ AG3})\,(inverses)$ $A\odot \iota(A)= A(a, k)\odot \iota(A(a, k))= A(a, k))\odot A(a^{-1},k)=$

\hspace*{6mm}$ A(1, k)= \theta(A(a,k))= \theta (A)=\iota(A)\odot A.$}
\end{Ex}
\begin{Ex} {\rm (\cite{ivan25a})~ Let $ B_{2}=\{ 0, 1\}~$ a set consisting of two elements. Consider the group $~({\bf Z}_{n}, + ) $ of integers modulo $ n, $ where $~{\bf Z}_{n} = \{ c_{1}=0,  c_{2}=1,\ldots, c_{n-1}=n-2,  c_{n}=n-1 \}.~$  We now define an almost groupoid  structure on
 the set $ G:=B_{2}\times  {\bf Z}_{n}~.$ Let $ G_{0} = \{ (a, c_{1}) | a\in B_{2}\}. $  Then $~ (G, \theta, \oplus, \iota, G_{0})~$ is an almost groupoid over $ G_{0}, $ where  the set of composable pairs $~ G_{(2)}~$ and  its structure functions are given by:\\
$\theta : G=B_{2}\times  {\bf Z}_{n}~ \rightarrow~ G_{0}, (a, c_{j})~\rightarrowtail \theta (a, c_{j}):=(a, c_{1});$\\
$ G_{(2)} = \{ ((a, c_{j}), (b, c_{k}))\in G^{2}~|~\theta (a, c_{j}) =\theta (b, c_{k})\} = \{ ((a, c_{j}), (b, c_{k}))\in G^{2}~|~ a=b \};$\\
$(a, c_{j})~\oplus~(a, c_{k}):=(a, c_{j}+ c_{k}),~ (\forall)~a \in B_{2}~, c_{j}, c_{k}\in {\bf Z}_{n}~$ and\\
$\iota(a, c_{j}):=(a, -c_{j}), (\forall)~a\in B_{2},~c_{j}\in  {\bf Z}_{n}~.$  
It is easy to verify that the conditions $ ({\bf AG1}) - {\bf AG3}) $ of Definition 4.1 hold.\\
First, we consider $~ x, y, z\in G ,~$ where $~x=(a_{1}, c_{i}), y=(a_{2}, c_{j}), z=(a_{3}, c_{k}). $  The product $ x\oplus y \oplus z $ is defined if and only if $~a_{1} = a_{2} = a_{3}.$ We have:\\
$ (1)~~~(x\oplus y)\oplus z= ((a_{1}, c_{i})\oplus (a_{1}, c_{j}))\oplus (a_{1}, c_{k})= (a_{1}, c_{i}+c_{j})\oplus (a_{1}, c_{k})= (a_{1}, c_{i}+ c_{j}+ c_{k}) =x\oplus (y\oplus z); $\\
$ (2)~~~x\oplus \theta(x)= (a_{1}, c_{i})\oplus \theta(a_{1}, c_{i}))= (a_{1}, c_{i})\oplus(a_{1}, c_{1})=(a_{1}, c_{i}+0)= x =\theta(x)\oplus x;$\\
$ (3)~~~x\oplus \iota(x)= (a_{1}, c_{i})\oplus \iota(a_{1}, c_{i}))= (a_{1}, c_{i})\oplus(a_{1}, -c_{i})=(a_{1}, 0)= \theta(x)= \iota(x)\oplus x.$\\
We conclude that $ G = B_{2}\times {\bf Z}_{n}~$ is an almost groupoid over $ G_{0}=\{ (0, c_{1}),~(1, c_{1})\}.~$  Observe that $~ (G, G_{0})~$ is a finite almost groupoid of order $~2n.$ }
\end{Ex} 
\begin{Rem} {\rm ~ In paper \cite{ivan25a},  three important constructions of new almost groupoids are presented, namely: disjoint union of two  almost groupoids, direct product  of two  almost groupoids  and semidirect product  of two  almost groupoids}
\end{Rem}
\begin{defn}
{\rm  {\it (i)~} Let $ (G, \theta, m, \iota, G_{0}) $ and $ (G^{\prime}, \theta^{\prime}, m^{\prime}, \iota^{\prime}, G_{0}^{\prime}) $ be two
almost groupoids. A {\it morphism of almost groupoids} or {\it almost groupoid
morphism} from $ (G, G_{0}) $ into $ (G^{\prime}, G_{0}^{\prime}) $ is a pair $(f, f_{0}) $
of maps, where $ f:G \to G^{\prime} $ and $ f_{0}: G_{0} \to
G_{0}^{\prime} $ such that the following conditions hold:

$(1)~~~f(m(x,y)) = m^{\prime}(f(x),f(y))~~~\forall (x,y)\in
G_{(2)};$

$(2)~~~\theta^{\prime}\circ f = f_{0} \circ
\theta. $

{\it (ii)~} An almost groupoid morphism $ (f, f_{0}): (G, G_{0}) \to (G^{\prime}, G_{0}^{\prime})$ such that $ f $ and $f_{0}$ are bijective maps, is called {\it isomorphism of almost groupoids}.}
\end{defn}
\begin{Ex} {\rm   Let $~ G = B_{2}\times {\bf Z}_{n}~$ be  the  almost groupoid given in Example 3.4. Consider the  group $~({\bf Z}_{n}, + ) .$ It is an almost groupoid. 
Define the map    $~f:  B_{2}\times {\bf Z}_{n}~\rightarrow~{\bf Z}_{n}~$ by $~f((a,c_{j})):=c_{j}~$ for all $~a\in B_{2}~$ and $~c_{j}\in {\bf Z}_{n}.$ This map is a morphism of almost groupoids. Indeed, for all $~(a, c_{j}), (a, c_{k}) \in B_{2}\times {\bf  Z}_{n}~$  
we have $~ f((a, c_{j})~\oplus~(a, c_{k})) =f((a, c_{j}+ c_{k}))= c_{j}+ c_{k}~$   and $~  f((a, c_{j}))+ f((a, c_{k})) =  c_{j}+ c_{k}.~$  Hence, $~ f((a, c_{j}))~\oplus~(a, c_{k})) = f((a, c_{j})) + f((a, c_{k})).$ }\\[-0.4cm]
\end{Ex}
\begin{Rem} {\rm ~{\it  The pair groupoid over a set}. Let $~X~$ be a nonempty set. Then $~ \Gamma = X \times X~$ is a groupoid with respect to rules:  $~\alpha(x,y) = (x,x),~ \beta(x,y) = (y,y),~$ the elements $~ (x,y)~$ and $~(y^{\prime},z)~$
are composable in $~\Gamma~\Leftrightarrow~y^{\prime} = y~$ and we take $~(x,y)\cdot (y,z) = (x,z)~$ and the inverse of $~(x,y)~$ is defined by $~(x,y)^{-1} = (y,x).~$
This groupoid  is called the {\it pair groupoid} associated to $~X~$ and it is denoted with $~{\cal PG}(X).~$  Its unit space is $~{\cal PG}_{0}(X)= \{ (x,x)~|~ x\in X\}.~$   The groupoid $~({\cal PG}(X), {\cal PG}_{0}(X))~$ {\bf is not an almost groupoid}, because 
$~\alpha(x,y)=(x,x)\neq \beta(x,y)=(y,y) ~$ for  $ x\neq y.$}\\[-0.5cm]
\end{Rem}
\begin{Rem} {\rm ~The interested reader can find more results regarding almost groupoid morphisms in paper \cite{miva23a}.}\\[-0.7cm]
\end{Rem}

\section{The most important substructures defined in almost groupoids}

\normalsize{}
\begin{defn}
{\rm  Let $(G ,\theta, m, \iota, G_{0})$ be an almost groupoid. A pair of nonempty subsets $(H,H_{0})$
where $ H\subseteq G $ and ${H_0}\subseteq G_{0} $, is called
\textit{almost subgroupoid} of $ G, $ if:

$(1.1)~~\theta (H)=H_{0};$

$(1.2)~~H $ is closed under multiplication and inversion,
that is:\\[0.1cm]
$(1.2.1)~(\forall)~x,y\in H$  such that $(x,y)\in G_{(2)}
\Longrightarrow ~x\cdot y\in H;$\\[0.1cm]
 $~(1.2.2)~(\forall)~x\in
H\Longrightarrow x^{-1}\in H.$}
\end{defn}

\begin{defn}~{\rm Let $(H,H_{0})$ be an almost subgroupoid of the almost groupoid $ (G,G_{0}).$

{\it(i)} $ (H, H_{0}) $ is said to be a {\it wide almost subgroupoid} of $ (G, G_{0} ), $ if $~H_{0}= G_{0},$   that is  $ H $  and $ G $  have the same units.

{\it(ii)} A wide almost subgrupoid $ (H,H_{0})$  of $ (G,G_{0}),$  is called {\it normal almost subgroupoid,} if for all $ g\in G $ and for all
 $ h \in H $ such that the product $ g\cdot h\cdot {g^ {-1}} $ is defined, we have $ g\cdot h\cdot {g^ {-1}}\in H .$}
\end{defn}
\begin{Ex} {\rm $~(i)~$ Each subgroup (resp., normal) of the group $ ~(G, \{e\}) ~$  is a wide almost subgroupoid (resp., normal) of the almost groupoid $ ~(G, \{e\}). ~$\\
$~(ii)~$ If $~(G, G_{0})~$ is an almost groupoid, then {\bf $ G_{0} $ is a wide almost subgroupoid of $ G $}.  We shall check that the conditions (1.1) and (1.2) in Definition 4.1 hold.\\
$(1.1)~$ Clearly, $~\theta (G_{0}) = G_{0}.$\\
$(1.2)~$ Let $ u,v \in G_{0} $ such that $ (u,v)\in G_{(2)}.~ $ Then $ \theta(u)=\theta(v) ~$ and it follows that $ u=v, $ since $ u, v\in G_{0}. $ Therefore, the product $ u\cdot v $ is defined if and only if $ u=v. $ By Proposition 3.1(i) and (ii), we have $ u\cdot u =u. $  Hence $ u\cdot u\in G_{0}. $ Also, from $ u\cdot u = u ~$ implies that $ u^{-1} =u $ and $ u^{-1}\in G_{0}. $\\
$~(iii)~ G_{0} $ {\bf is a normal almost subgroupoid of the almost groupoid}  $ (G, G_{0}). $  For this, we consider $ g \in G $ and $ u\in G_{0} $ such that  $ g\cdot u \cdot g^{-1} $ is defined in $ G. $
Then $ \theta(g) = \theta(u)= u, $  since $ u\in  G_{0}.~$  But, $ \theta (g\cdot u\cdot g^{-1}) = \theta (g)= u. ~$  Hence, $~g\cdot u\cdot g^{-1} \in G_{0}. $}
\end{Ex}
\begin{Prop} (\cite{miva23a}) Let $ (G, G_{0}) $  be an almost groupoid.  Then:\\
$(i)~~$ The isotropy group  $~G(u), u\in G_{0}~$ is an almost subgroupoid of $ G. $\\
$(ii)~~$  The set $~Is(G) = \cup_{u\in G_{0}}G(u)\subset G~$ is a normal almost subgroupoid of $~G.~$
\end{Prop}
The pair $~(Is(G), G_{0}) $ is called the {\bf isotropy almost subgroupoid} of $~G. $\\[-0.2cm]

If $~(G, \theta, m, \iota, G_{0}) $ is an almost groupoid and $ (H, H_{0}), ~(K, K_{0}) $  are almost subgroupoids of $ G, $ we define the set:\\[-0.3cm]
$$~~~~~~~~~~~~~\begin{array}{l}
 H\coprod K:= H\cup K ~~~\hbox{when}~~~ H\cap K=\emptyset.
\end{array}~~~~~~~~~~~~~~~~~~~~~~~~~~~~~~~~~~~~~~~~~~~~~ (2)$$

The set $~H\coprod K $ is called the {\it disjoint union} of the almost subgroupoids $ H $ and $ K. $

In general, if $~\{(H_{i}, H_{0,i})~|~i\in I\}~$ is a family of almost subgroupoids  such that $~H_{i}\cap H_{j}=\emptyset, ~(\forall)i,j\in I, i\neq j, $  then the
set  $~H = \cup_{i\in I} H_{i}~$ is called the {\it disjoint union} of the almost subgroupoids $ H_{i}, i\in I~ $ and it is denoted by  $~\coprod_{i\in I} H_{i}. $

\begin{Prop}~ {\it Let $ (G, G_{0}) $  be an almost groupoid.\\
$(i)~$ If $~(H, H_{0}), ~(K, K_{0}) $  are almost subgroupoids of $ G, $ then $~H\coprod K~$ is an almost subgroupoid of $ G,~$ having $ H_{0}\cup K_{0} $
as units set.\\
$(ii)~$ If  $~ \{(H_{i}, H_{0,i})~|~i\in I \}~$ is a disjoint family of almost subgroupoids of $~G, $ then  $~\coprod_{i\in I} H_{i} $ is an almost subgroupoid  of $~G.$}
\end{Prop}
{\bf Proof.} $ (i)~$ Clearly, $~\theta(H \cup K) = \theta(H)\cup \theta(K) = H_{0}\cup K_{0}, $ since $ \theta(H)=H_{0}, \theta(K)= K_{0}~$ and $~H\cap K =\emptyset. $
 If $ x , y\in H\cup  K, $ then $ x, y \in H~$ or $~ x,y \in K.~$  If  $~(x,y) \in G_{(2)}, $ then $~(\exists) x y $ and it follows that $ x y\in H~$ or $ xy\in K, $  since $ H $ and $ K $ are almost
 subgroupoids of $ G.$  Hence $~xy\in H\cup K.~$  Hence, the condition (1.2.1) in Definition 4.1 is satisfied.

Also, if  $ x\in H\cup K, $ then $ x \in H $ or $ x\in K. $  Then  $ x^{-1}\in H~$ or $ x^{-1}\in K~$  and $ x\in H\cup K.~$ Hence , the condition (1.2.2) in Definition 4.1 holds.
We conclude that, $ H\coprod K $ is an almost subgroupoid of $ G. $\\
$ (ii)~$ Denote $ H:= \coprod_{i\in I} H_{i}. $ In this case, two elements $ x, y \in H~$ may be composed if and only if they lie in the same almost subgroupoid $~H_{i}.~$
Using the same reasoning as in the proof of statement $(i)$, we prove that $~\coprod_{i\in I} H_{i} $ is an almost subgroupoid  of $ G.~$ Its units set
$~H_{0}:= \cup_{i\in I} H_{i,0},~$ where $~H_{i,0}~$ is the units set of $~H_{i}.~\hfill\Box$

\begin{Cor}~{\it Let $ (G, G_{0}) $  be an almost groupoid. If  $~\{(H_{i}, H_{0,i})~|~i\in I\}~$ is a disjoint family of almost subgroupoids of $~G $ such that
$~\cup_{i\in I} H_{i,0} = G_{0},~$ then  $~(\coprod_{i\in I} H_{i}, G_{0}) $ is a wide almost subgroupoid  of $~G.$}
\end{Cor}
{\bf Proof.}  This corollary follows immediately from definition of the wide almost groupoid and Proposition 4.2(ii).\hfill$\Box$

If $~(G, \theta, m, \iota, G_{0}) $ is an almost groupoid and $ a\in G,~$  we define the sets:\\[-0.3cm]
$$~~~~~~~~~~~\begin{array}{l}
 C(a):= \{ g\in G(\theta(a)) ~|~ g\cdot a = a\cdot g, ~(\forall) g\in G(\theta(a)) \}.
 \end{array}~~~~~~~~~~~~~~~~~~~~~~~~~~~~~~~~~~~~~~~~~ (3)$$
$$~~~~~~~~~~~~~~~\begin{array}{l}
 Z(G):= \{ a\in Is(G)~|~ x\cdot a = a\cdot x, (\forall) x\in G~ \hbox{such that}~ \theta(x)=\theta(a) \}.
 \end{array}~~~~~~~~~~~~~~~~~~~~~~~~(4) $$

$ C(a) $, resp. $~Z(G),~$ are called the {\it centralizer} of  $~ a $ in $ G ,$  resp. the {\it center} of  $~G.$

\begin{Prop}~{\it  Let $ (G, G_{0}) $  be an almost groupoid.  Then:\\
$(i)~~$ The centralizer   $~C(a)~$ of $ a\in G $ is an almost subgroupoid of $ G. $\\
$(ii)~~~Z(G) = \cup_{u\in G_{0}} Z(G(u))\subset G.~$\\
$(iii)~~$  The center $~Z(G)~$ of $~G~$ is a normal almost subgroupoid of $~Is(G).~$}
\end{Prop}
{\bf Proof.} $ (i)~$ Note that $~\theta(a)\in C(a)\subset G(\theta(a)), $  since $~g\cdot \theta(a) = \theta(a)\cdot g $ for all $ g\in G(\theta(a)). $  Let $ x, y\in C(a). $  Then $ x, y \in G(\theta(a)),~g\cdot x = x\cdot g ~$ and $ g\cdot y = y\cdot g. $ Since $ G(\theta(a)) $ is a group, implies  $~(\exists)~ g\cdot (x\cdot y) $ and $~ (x\cdot y )\cdot g. $  We have\\
$~ g\cdot (x\cdot y)= (g\cdot x)\cdot y = (x\cdot g)\cdot y = x\cdot (g\cdot y)= (x\cdot y)\cdot g .$  Hence, $~x\cdot y\in C(a).$ Also, if $ x\in C(a), $ then  is easy to prove that\\
$~~~~~g\cdot x = x\cdot g~~~\Longleftrightarrow ~~~ g\cdot x^{-1} = x^{-1}\cdot g. $\\
Hence, $~x^{-1} \in C(a).~ $  Therefore, $~C(a) ~$ is an almost subgroupoid of $~G. $\\
 $(ii)~~$ Clearly,  $~G_{0}\subseteq Z(G)\subseteq Is(G). $\\
 $(iii)~~$  The set $~Z(G) = \cup_{u\in G_{0}} Z(G(u))\subset G~$ is a normal almost subgroupoid of $~Is(G).~$
Let $ x,y \in Z(G) ~$ and suppose that $ (\exists) x\cdot y. $ Then, $ x\in Z(G(u))~$ and $~y\in Z(G(v)~$ for some $ u,v\in G_{0}. $ Moreover, $ \theta(x)=u~$ and $~\theta(y) = v~$ and now since
$\theta(x)=\theta(y), $ it follows that $ u=v.$ Thus, $ y\in G(u)~$ and hence $ x\cdot y\in Z(g(u))\subset Z(G). $  Also, if $ x\in Z(G), $ then $ x\in Z(G(u))~$ for some $ u\in G_{0}~$ and since $ Z(G(u))~$
is a subgroup of $~G(u), $  implies  $ x^{-1} \in Z(G(u))\subset Z(G). $ Therefore, $~Z(G)~$ is a wide  almost subgroupoid.\\
Moreover, let $ g\in Is(G) $ and $ h\in Z(G)~$ such that $ \theta(h)=\theta(g). $ Then, $ g\in G(u) $ for some $~u\in G_{0}~$ and $~\theta(g)=u. $ Then $ \theta(h)=u $ and $~h\in Z(G(u)). $
Since $(\exists)~ g\cdot h \cdot g^{-1}, $ we have that $~g\cdot h \cdot g^{-1}= (g\cdot h) \cdot g^{-1})= (h\cdot g \cdot g^{-1}= h\cdot (g \cdot g^{-1})=h\cdot \theta(g)=h\cdot \theta(h)=h\in Z(G).~$
Hence, $~Z(G)~$ is a normal almost subgroupoid of $~Is(G).\hfill\Box$

If $~(G, \theta, m, \iota, G_{0}) $ is an almost groupoid and $ (H, H_{0}), ~(K, K_{0}) $  are almost subgroupoids of $ G, $ we define the set:\\[-0.4cm]
$$~~~~~~~~~~~~~~~\begin{array}{l}
 H K:= \{ h\cdot k ~|~\theta(h)=\theta(k), ~h\in H, k\in K \}.
 \end{array}~~~~~~~~~~~~~~~~~~~~~~~~~~~~~~~~~~~~~~~~ (5)$$

The following propositions extends to the context of almost groupoids several elementary results for groups.
\begin{Prop}~{\it Let $ (G, G_{0}) $  be an almost groupoid.  If $~(H, H_{0}) $ and $ (K,K_{0})~$ are wide almost subgroupoids such that $ H K =K H,$ then $ H K $ is a wide almost subgroupoid of $ G.$}

{\bf Proof.} Note that $~G_{0} \subseteq H K, ~$ since $\theta(g)=\theta(g)\cdot \theta(g)\in H K $ for each $ g\in G. $ Hence $~ H K\neq \emptyset. $  It follows
$~\theta(H K) = G_{0}.~$ Suppose  $~H K=K H. $  Let $~h_{1}, h_{2} \in H $ and $~k_{1}, k_{2} \in K. $  If $ x_{1} , x_{2}\in H K, $ then $ x_{1}= h_{1}k_{1} ~$ and
$ x_{2}= h_{2}k_{2} ~$ such that $ \theta(h_{1})= \theta(k_{1}) ~ $  and $ \theta(h_{2}) = \theta(k_{2}).~$ If $~(\exists) x_{1} x_{2}, $  then\\
 $x_{1}x_{2}= (h_{1} k_{1})(h_{2}k_{2})=h_{1} (k_{1}h_{2})k_{2}= h_{1} (h^{\prime} k^{\prime}) k_{2}= ( h_{1} h^{\prime}) ((k^{\prime}k_{2})= h_{3} k_{3}\in H K, $ such that $ \theta(h_{1}h^{\prime})=\theta(k^{\prime} k_{2}). $  Hence, the condition (1.2.1) in Definition 4.1 is satisfied. Finally, if  $ x\in H K, $ then $ x = h k $ with $ h\in H , k\in K  $ and $ \theta(h) = \theta (k). $ Then  $ x^{-1}= (h k)^{-1})= k^{-1} h^{-1}\in K H = H K.~$ Thus, the condition (1.2.2) in Definition 4.1 holds. We conclude that, $ HK $ is a wide almost subgroupoid.\hfill$\Box$
\end{Prop}
\begin{Prop}~{\it Let $ (G, G_{0}) $  be an almost groupoid   and $~H_{i\in I} $  a family of almost subgroupoids of $~G. $ Then:
$(i)~~$ If $~\cap_{i\in I} H_{i} \neq \emptyset, $ then $~\cap_{i\in I} H_{i}~$ is an almost subgroupoid of $~G. $
$(ii)~~$ If $~H_{i} $ is a wide (resp. normal) almost subgroupoid for each $ i\in I, $ then $~\cap_{i\in I} H_{i}~$ is a wide (resp. normal) almost subgroupoid of $~G.$}\\[-0.5cm]
\end{Prop}
Suppose now $~S~$ is a subset of an almost groupoid $~(G, G_{0} ).~$ Consider the family of all almost subgroupoids of $~G~$ that contains $~S.~$ This family is not empty, since
 it clearly includes the groupoid $~G~$ itself. By Proposition 4.5, the intersection of all almost subgroupoids in this family is also an almost subgroupoid of $~G~$
 and it clearly contains $~S.~$ It is then the smallest almost subgroupoid of $~G~$ which contains $~S.~$ We denote it by $~< S >~$
and call it the {\it generated almost subgroupoid} by $~S.~$

If $~x_{i}\in G,~i=\overline{1,n},$ then:\\
{\it the product $~x_{1}\cdot x_{2}\ldots \cdot x_{n}~$ is defined in $~G~~\Leftrightarrow ~~ \theta (x_{j}) = \theta(x_{j+1})~~(\forall)~~j = \overline{1, n-1}.~$}

It is easy to check the following
\begin{Prop}~{\it If $~S~$ is a nonempty set of an almost groupoid $~(G, G_{0},~$  then the almost subgroupoid $~< S >~$ generated by $~S~$ consists of all products of the
form\\
$~~~~~~~~x_{1}^{\varepsilon_{1}}x_{2}^{\varepsilon_{2}}\ldots x_{n}^{\varepsilon_{n}},~$ where $~x_{i} \in S,~\varepsilon_{i} = \pm 1,~ i = \overline{1,n}~$  for each $~n\in {\bf N}^{*}.$ }

In particular, the almost subgroupoid generated  by $ S=\{a\} \subset G, $  is given by:\\[-0.5cm]
\end{Prop}
$$~~~~\begin{array}{l}
< a> =\{ a^{n} ~|~(\forall)~n\in {\bf Z}\},~\hbox{where}~~ a^{0}=\theta (a),~~a^{n}= a^{n-1}\cdot a,~~a^{-n}=(a^{-1})^{n},~n\in {\bf N}^{\ast}\\
\end{array}~~~~(6)$$
and it is called the {\it cyclic almost subgroupoid generated by} $ a\in G. $

\begin{Rem}~{\rm Let $~(G, G_{0}~$ be an almost groupoid  and $~a\in G.$  Then:\\
$(i)~~~ < a >\subset G(\theta(a))~~$ and $~<a>~$ is a subgroup of the isotropy group $~G(\theta(a)). $\\
$(ii)~$ If $~(\exists)~ n\in {\bf N}, ~n\geq 2~ $ such that $~a^{n} =\theta(a),$ then  $ < a > $ is a cyclic subgroup of order $ n~ $ and
$~<a>  =\{ \theta (a), a, a^{2}, ..., a^{n-1} \}.$}\\[-0.4cm]
\end{Rem}
\begin{Ex}~{\bf The almost groupoid associated to  $~{\bf Z}_{6} .$ } {\rm~Consider the group $~({\bf Z}_{6}, + ) $ of integers modulo $ 6, $ where $~{\bf Z}_{6} = \{ 0, 1, 2, 3, 4, 5 \}~$ and the subgroup $~H =\{ 0, 2, 4\} \subset {\bf Z}_{6}. $\\
Let $~ G:= H \times {\bf Z}_{6}~$ and $ G_{0}:=\{ (0,0), (0,1), (0,2), (0,3), (0,4), (0,5)\} \subset G. $  Then
$ (G, \theta, \ominus, \iota, G_{0})~$ is an almost groupoid over $ G_{0}, $ where  the set of
composable pairs $~ G_{(2)}~$ and  its structure functions are given by:\\
$\theta : G=H\times {\bf Z}_{6}~ \rightarrow~ G_{0}, (a,b)~\rightarrowtail \theta (a,b):=(0,b-a);$\\
$ G_{(2)} = \{ ((a,b), (c, d))\in G^{2}~|~\theta (a,b) =\theta (c, d)\} = \{ ((a,b), (c, d))\in G^{2}~|~b-a=d-c\};$\\
$(a,b)~\ominus~(c,d):=(a+c, b+c),~ (\forall)~a, b, c, d \in {\bf Z}_{6}~~$ and\\
 $\iota(a,b):=(-a,b-2 a), (\forall)~a,b\in {\bf Z}_{6}. $\\[-0.4cm]

It is easy to verify that the conditions $ ({\bf AG1}) - {\bf AG3}) $ of Definition 3.1 hold. For this, let $~ x, x_{1}, x_{2}\in G ,~$ where $~x=(a, b), x_{1}=(a_{1}, b_{1}), x_{2}=(a_{2}, b_{2}) $
with $~a, a_{1}, a_{2} \in H $ and $~b, b_{1}, b_{2}\in {\bf Z}_{6}.~$  The element  $ x\ominus x_{1} \ominus x_{2} $ is defined $~~\Leftrightarrow~~b-a= b_{1}- a_{1} = b_{2}-a_{2}.$ We have:\\
$ (1)~~~(x\ominus x_{1})\ominus x_{2}= ((a, b)\ominus (a_{1}, b_{1}))\ominus (a_{2}, b_{2})= (a +a_{1}, b +a_{1})\ominus (a_{2}, b_{2})= \\
(a +a_{1}+a_{2}, b_{1}+ a_{1}+ a_{2}) =x\ominus (x_{1}\ominus x_{2}); $\\
$ (2)~~~x\ominus\theta(x)= (a, b)\ominus \theta(a, b)= (a, b)\ominus(0, b-a)=(a, b)= x =\theta(x)\ominus x;$\\
$ (3)~~~x\ominus \iota(x)= (a, b)\ominus \iota(a, b))= (a, b)\ominus(-a, b-2a)=(0, b-a)= \theta(x)= \iota(x)\ominus x.$\\[-0.4cm]

We conclude that $ G = H\times {\bf Z}_{6}~$ is an almost groupoid over $ G_{0}.~$  Observe that $~G= \{ u_{i}, p_{j}\in H\times {\bf Z}_{6}~|~ i=\overline{1,6}, j=\overline{1,12}\}~$ is a finite almost groupoid of order $ 18, $ where:\\
$~u_{1}=(0,0),~ u_{2}=(0,1),~ u_{3}=(0,2),~ u_{4}=(0,3),~ u_{5}=(0,4),~ u_{6}=(0,5),$\\
$~p_{1}=(2,0),~ p_{2}=(2,1),~ p_{3}=(2,2),~ p_{4}=(2,3),~ p_{5}=(2,4),~ p_{6}=(2,5),~$\\
$~p_{7}=(4,0),~ p_{8}=(4,1),~ p_{9}=(4,2),~ p_{10}=(4,3),~ p_{11}=(4,4),~ p_{12}=(4,5).$\\[-0.4cm]

We now give some examples for completing the following three  tables:\\
$\theta(u_{4}) =\theta(0,3)=(0,3-0)=u_{4};~~~~~~~~~~~~~~~\theta(p_{7})=\theta(4,0)= (0, 0-4)=(0,2)=u_{3}.$\\
$\iota(u_{6})=\iota(0,5)=(0, 5-2\cdot 0)=(0, 5)= u_{6};~~
\iota(p_{4})=\iota(2,3)=(-2, 3-2\cdot 2)=(4,5)=p_{12};$\\
$ G(u_{3}) = \{ u_{3}, p_{5}, p_{7}\}; ~~~~~~~~ G(\theta(p_{3}))= G(u_{1})=\{ u_{1}, p_{3}, p_{11}\}.$\\
$u_{6}\ominus p_{10}= (0,5)\ominus (4,3)= (0+4, 5+4)=(4,3)=p_{10};$\\
$p_{4}\ominus p_{4}= (2,3)\ominus (2,3)= (2+2, 3+2)=(4,5)=p_{12};$\\
$p_{8}\ominus p_{6}= (4,1)\ominus (2,5)= (4+2, 1+2)=(0,3)=u_{4}.$\\[-0.4cm]

 The element $~p_{1}\ominus p_{10} ~$ is not defined, since $ \theta(p_{1})=\theta(2,0)= u_{5},~ \theta(p_{10})=\theta(4,3)= (0,3-4)=(0,5)= u_{6}~$  and $~u_{5}\neq u_{6}. $
{\small
$$\begin{array}{|r|c|c|c|c|c|c|c|c|c|c|c|c|c|c|c|c|c|c|} \hline
       g    & u_{1} & u_{2} & u_{3} & u_{4} & u_{5} & u_{6} & p_{1} & p_{2} & p_{3} & p_{4} & p_{5} & p_{6} & p_{7} & p_{8} & p_{9} & p_{10} & p_{11} & p_{12}\\ \hline
  \theta(g) & u_{1} & u_{2} & u_{3} & u_{4} & u_{5} & u_{6} & u_{5} & u_{6} & u_{1} & u_{2} & u_{3} & u_{4} & u_{3} & u_{4} & u_{5} &  u_{6} &  u_{1}& u_{2} \cr \hline
     \iota(g)  & u_{1} & u_{2} & u_{3} & u_{4} & u_{5} & u_{6} & p_{9} & p_{10} & p_{11} & p_{12} & p_{7} & p_{8} & p_{5} & p_{6} & p_{1} & p_{2}&p_{3} & p_{4}\cr \hline
\end{array}$$}
The multiplication operation $~\ominus $ on  $~G= H\times {\bf Z}_{6}~$ is given in the following table:
{\small
$$\begin{array}{|r|c|c|c|c|c|c|c|c|c|c|c|c|c|c|c|c|c|c|} \hline
\ominus  & u_{1} & u_{2} & u_{3} & u_{4} & u_{5} & u_{6} & p_{1} & p_{2} & p_{3} &
p_{4} & p_{5} & p_{6} & p_{7} & p_{8} & p_{9} & p_{10} & p_{11} & p_{12}\\ \hline
u_{1}  & u_{1} &  &  &  &  &  &  &  &p_{3}  &  &  &  &  & &  &  & p_{11} &\cr \hline
u_{2}  &  & u_{2} &  &  &  &  &  &  &  & p_{4}  &  &  &  &  & &  & & p_{12}\cr \hline
u_{3}  &  &  & u_{3} &  &  &  &  &  &  &  & p_{5} &  & p_{7}  &  & &  & &\cr \hline
u_{4}  &  &  &  & u_{4} & & & &  &   &  &  & p_{6} & & p_{8}  &  & &  & \cr \hline
u_{5} &  &  &  &  & u_{5} & & p_{1} & & &  &  &  &  &  & p_{9} &  & &\cr \hline
u_{6}  &  &  &  &  &  &  u_{6} & & p_{2} & & & &  &  &  & & p_{10} & &\cr \hline
p_{1} &  &  &  &  & p_{1} &  & p_{9} & & & & & & &  & u_{5} &  & &\cr \hline
p_{2}  &  &  &  &  &  & p_{2} &  & p_{10} &  &  &  &  &  &  & & u_{6} & &\cr \hline
p_{3} & p_{3} &  &  &  & &  & & & p_{11} & & &  &  &  & &  & u_{1} &\cr \hline
p_{4}  &  & p_{4} & &  & & & &  &  & p_{12} & & & & & &  & & u_{2}\cr \hline
p_{5} &  &  & p_{5} & &  &  & & & & &  p_{7} &  & u_{3} & &  & & &\cr \hline
p_{6} &  & & & p_{6} &  & & &  &  & & & p_{8} & & u_{4} & &  & &\cr \hline
p_{7} & &  & p_{7} &  &  & & & & &  & u_{3} & & p_{5} & & &  & &\cr \hline
p_{8} &  & &  & p_{8} &  &  &  & & & & & u_{4} & & p_{6} &  &  & &\cr \hline
p_{9} & &  &  &  & p_{9} &  & u_{5} & & & & & & &  & p_{1} &  & &\cr \hline
p_{10} & &  &  &  &  & p_{10}  &  & u_{6} & & & & & &  &  & p_{2}  & &\cr \hline
p_{11} & p_{11} &  &  &  &  &  &  & & u_{1} & & & & &  &  &  & p_{3} &\cr \hline
p_{12} & & p_{12} &  &  &  &  &  & & & u_{2} & & & &  &  &  & & p_{4} \cr \hline
\end{array}.~$$}

The centralizer $~ C(p_{1}) ~$ of $ p_{1} $ in the almost groupoid $ G $ is\\
$C(p_{1}) = \{ g\in G(\theta(p_{1}) | g\ominus p_{1} =p_{1}\ominus g \} =  \{ g\in G(u_{5}) | g\ominus p_{1} =p_{1}\ominus g \} =
\{ u_{5}, p_{1}, p_{9} \} =  G(u_{5}).~$ Also, we have  $~C(u_{1}) =G(u_{1}). $\\
Finally, it is easy to observe that the isotropy almost subgroupoid $ Is(G) $ and the center $~Z(G) ~$ of $ G~$ satisfy the relation $ Is(G) = Z(G) = G. $}
\end{Ex}

\vspace*{0.2cm}

Author's adress\\[0.2cm]
\hspace*{0.7cm}West University of Timi\c soara. Seminarul de Geometrie \c si Topologie.\\
\hspace*{0.7cm} Teacher Training Department. Timi\c soara.  Romania.\\
\hspace*{0.7cm} Bd. V. P{\^a}rvan,no.4, 300223, Timi\c soara, Romania\\
\hspace*{0.7cm}E-mail: mihai.ivan@e-uvt.ro\\

\end{document}